\documentclass[10pt, reqno]{amsart}

\usepackage[T1]{fontenc}
\usepackage[utf8]{inputenc}
\usepackage[english]{babel}

\usepackage{latexsym, amsmath, amssymb, amsthm, mathtools, mathrsfs}
\usepackage[hypertexnames=false]{hyperref}
\hypersetup{
    hidelinks,
    colorlinks,
    linkcolor={red!50!black},
    citecolor={blue!50!black},
    urlcolor={blue!80!black},
}

\usepackage{enumerate}

\mathtoolsset{showonlyrefs=true}

\usepackage[textsize=tiny]{todonotes}
\usepackage{xcolor}

\usepackage[
    a4paper,
    left=1in,
    right=1in,
    top=1.3in,
    bottom=1.3in
]{geometry}

\def\R{\mathbb R}
\def\N{\mathbb N}

\newcommand{\X}{X}
\newcommand{\Y}{Y}
\newcommand{\Z}{\mathbb Z}
\def\dim{\mathrm {dim}}
\def\h{\widehat}
\def\t{\widetilde}
\def\l{\overline}
\def\la{\langle}
\def\ra{\rangle}

\def\cH{\mathcal H}

\def\cO{\mathcal O}
\def\cS{\mathcal S}

\def\distr{\mathcal{H}}
\def\MCP{\mathrm {MCP}}
\def\Cut{\mathcal{C}}
\def\loc{\mathrm {loc}}
\def\Abn{\mathrm{Abn}}

\newtheorem{theorem}{Theorem}

\newtheorem{lemma}[theorem]{Lemma}
\newtheorem{proposition}[theorem]{Proposition}
\newtheorem{corollary}[theorem]{Corollary}

\newtheorem{definition}[theorem]{Definition}

\newtheorem{remark}[theorem]{Remark}

\title[Two-step ideal sR manifold with no MCP]{A two-step ideal sub-Riemannian structure with no measure contraction properties}
\date{\today}

\author[L. Rizzi]{Luca Rizzi}
 \address[L.~Rizzi]{SISSA, Via Bonomea 265, 34136 Trieste (IT)}
 \email{lrizzi@sissa.it}

\author[Y. Zhang]{Ye Zhang}
 \address[Y.~Zhang]{SISSA, Via Bonomea 265, 34136 Trieste (IT)}
 \email{yezhang@sissa.it}

\subjclass[2020]{53C17, 53C21}
\keywords{measure contraction property, sub-Riemannian manifold}

\begin{document}

\begin{abstract}
We construct a compact, equiregular, ideal sub-Riemannian manifold of step~$2$ such that, for every smooth positive measure, the $\MCP(K,N)$ fails for all $K\in\R$ and $N\in (1,\infty)$. This shows that the real-analyticity assumption in the measure contraction theorem of Badreddine and Rifford in \cite{BR20} cannot be replaced by smoothness. Moreover, our structure is ideal, i.e.\ it admits no non-trivial abnormal minimizing geodesics. Although failures of the measure contraction property for ideal structures were recently obtained in the higher-step setting, our construction shows that the phenomenon can already occur on compact, equiregular, ideal structures \emph{of step~$2$}. The proof exploits a differential consequence of the measure contraction property, namely a uniform upper bound for the sub-Laplacian of the squared distance near its base point, and constructs a structure for which this quantity is unbounded.
\end{abstract}

\maketitle

\section{Introduction}

The works of Sturm \cite{S060,S06} and Lott--Villani \cite{LV09} introduced, through optimal transport, synthetic curvature-dimension conditions for metric measure spaces. These conditions are stable under measured Gromov--Hausdorff convergence and retain many of the geometric and functional-analytic consequences of lower Ricci curvature bounds in the nonsmooth setting. However, despite arising as measured Gromov--Hausdorff limits of Riemannian manifolds, sub-Riemannian manifolds do not satisfy these curvature-dimension conditions unless they are Riemannian \cite{J09,J10,AS20,J21,RS23,MR23,NP25}.

A weaker synthetic curvature bound, known as the measure contraction property ($\MCP$ for short), was independently introduced by Sturm \cite{S06} and Ohta \cite{O07}. Whereas the curvature-dimension condition controls optimal transport between arbitrary absolutely continuous probability measures, the measure contraction property controls contraction toward a fixed point. It nevertheless retains several important geometric consequences of curvature-dimension bounds and has therefore emerged as a natural synthetic curvature condition in sub-Riemannian geometry. In the pioneering work \cite{J09}, Juillet proved that the Heisenberg groups satisfy the measure contraction property. Since then, its validity has been established for several classes of sub-Riemannian structures; see, among others, \cite{R13,AL14,LLZ16,R16,BR18,BR19,RS19,BGKT19,BR202,M-QCD,GN24,BGRV25,Z25,BMR26} and the references therein.

A general result in this direction was obtained by Badreddine and Rifford \cite{BR20}, who proved that the $\MCP(0,N)$ holds for some finite $N>1$ on every compact, real-analytic sub-Riemannian manifold of step~$2$. They also established an analogous result for Lipschitz Carnot groups, namely Carnot groups whose  squared distance function is locally Lipschitz in charts away from the diagonal. Since every step-$2$ sub-Riemannian structure enjoys this local Lipschitz regularity (in fact, even along the diagonal \cite{AL09}), it was natural to ask whether real-analyticity in their compact-manifold result could be replaced by smoothness. This question remained open because real-analyticity enters only through a technical finite-order estimate, namely \cite[Lemma~2.11]{BR20}. Our main result gives a negative answer.

\begin{theorem}\label{tgol}
There exists a compact, equiregular, ideal sub-Riemannian manifold of step~$2$ such that, for every smooth positive measure, the corresponding metric measure space fails to satisfy the $\MCP(K,N)$ for every $K\in\R$ and every $N\in(1,\infty)$.
\end{theorem}

Recall that an ideal sub-Riemannian manifold is complete and admits no non-trivial abnormal minimizing geodesics. Although failures of the measure contraction property within the ideal class were recently obtained in the higher-step setting in \cite{BR25}, our construction shows that this phenomenon already occurs at step~$2$. This is in sharp contrast with the counterexamples recently discovered  in \cite{BR25}, where failure of the measure contraction property is linked with the lack of local Lipschitzness for the metric tangents, which can happen only for step greater than or equal to $3$.

The result is also complementary to the positive theorem for compact sub-Riemannian manifolds with fat distributions obtained in \cite[Corollary~9.24]{BMR26}. Fat structures are ideal, but fatness is a considerably stronger condition. Theorem~\ref{tgol} therefore shows that the latter result cannot be extended to general ideal step-$2$ structures without imposing some additional geometric or analytic condition.

In the broader sub-Finsler setting, failures of all measure contraction properties were exhibited in \cite{BT25}, where the authors proved that the $\ell^p$-sub-Finsler Heisenberg group does not satisfy the $\MCP(K,N)$ for any $K\in\mathbb R$ and $N\in (1,\infty)$ whenever $p\in(2,\infty]$. See also \cite{BMRT24-sufFinslerHeisenberg} for more general results. These examples are not sub-Riemannian and therefore do not address the question considered here.

\subsection{Sketch of the argument.} Let us briefly describe the mechanism behind our construction. On an ideal sub-Riemannian manifold, the measure contraction property admits a differential characterization in terms of the sub-Laplacian of the squared distance. In particular, the $\MCP$ implies a uniform upper bound for $\Delta_\mu d_x^2$ at smooth points approaching the base point $x$. We construct a smooth, equiregular, step-$2$ structure on $\R^4$ and a sequence of smooth points $x_j\to 0$ such that
\[
\Delta_\mu d_0^2(x_j)\longrightarrow+\infty
\]
for every smooth positive measure $\mu$. This contradicts the differential characterization of the $\MCP(K,N)$ for every choice of the parameters.

The structure depends on a single smooth function $A$. The relevant horizontal second derivative of $d_0^2$ can be computed explicitly and is governed by the quotient
\[
\frac{\left(\int_0^1 A(\tau s) ds\right)^2}
{\int_0^1 A(\tau s)^2 ds}.
\]
We construct $A$ and a sequence $\tau_j\to0^+$ for which this quotient converges to~$1$ from below. In other words, the Cauchy--Schwarz inequality becomes asymptotically sharp along the rescalings $A(\tau_j\cdot)$, while the functions remain nonconstant. The resulting degeneracy forces the horizontal second derivative, and hence $\Delta_\mu d_0^2$, to diverge. The function $A$ can moreover be chosen periodic, allowing the local construction to be compactified by quotienting with respect to a cocompact lattice.

The real-analyticity assumption in \cite{BR20} prevents precisely this type of infinite-order degeneracy. One might expect that it could be replaced by a suitable finite-order condition formulated in terms of ample curves and their geodesic flags \cite{ABR-curvature}. Theorem~\ref{tgol} shows that no such condition follows from smoothness, equiregularity, step~$2$, even for ideal structures. In particular, there appears to be no characterization depending only on the step and the absence of abnormal minimizing geodesics that guarantees the measure contraction property.

\subsection*{Use of AI}
In accordance with the Leiden Declaration on Artificial Intelligence and Mathematics, we disclose the following use of generative AI. After being provided with a counterexample to \cite[Lemma 2.11]{BR20}, independently constructed by the authors, OpenAI’s ChatGPT (GPT-5.5 Pro) suggested the structure \eqref{framecont}, which was subsequently incorporated into the construction presented here. The authors independently developed the resulting construction, verified all proofs and computations, and modified the arguments where necessary. All final mathematical judgments and decisions were made by the authors, who assume full responsibility for the correctness and content of the manuscript.

\subsection*{Funding} This work was supported by the European Research Council (ERC) under the European Union’s Horizon 2020 research and innovation programme (grant agreement GEOSUB, No.\ 945655), and by the IRP project GEOSUBMAN by INSMI-CNRS. The authors also acknowledge the support from INdAM.

\section{Preliminaries}\label{sP}

\subsection{Sub-Riemannian manifolds}

We recall some basic facts about sub-Riemannian manifolds. We refer the reader to \cite{M02, R14, ABB20} for more details. Note that the definition given below actually includes the classical constant-rank case, see for example \cite[\S~3.1.4]{ABB20}.

\medskip

A \emph{sub-Riemannian structure} on a smooth, connected $n$-dimensional manifold $M$, where $n \ge 2$, is defined by a set of $m (\ge 2)$ global smooth vector fields $\X_{1}, \ldots, \X_{m}$, called a \emph{generating frame}. We always assume the \emph{bracket-generating} condition, i.e., the vector fields $\X_{1},\ldots, \X_{m}$ and their iterated Lie brackets at $x$ generate the tangent space $T_x M$, for all $x\in M$. The \emph{distribution} is the possibly rank-varying family of subspaces of the tangent spaces spanned by the vector fields at each point
\begin{equation}
\distr_{x}\coloneq\mathrm{span}\{\X_{1}(x),\ldots,\X_{m}(x)\}\subset T_{x}M,\qquad \forall\, x\in M.
\end{equation}
Similarly, for $i\ge 1$, we can define the \emph{iterated distributions} by
\[
\distr^{i}_{x}\coloneq\mathrm{span}\{[\X_{k_{1}},[\ldots,[\X_{k_{j-1}},\X_{k_{j}}]]](x) : \,  1\le j\le i, \, 1 \le k_{1}, \ldots, k_{j} \le m \} \subset T_xM.
\]
The \emph{step} of the distribution at $x$, which we denote by $s_x$, is the minimal number $r \ge 1$ such that $\distr^{r}_{x}=T_{x}M$. The \emph{step} of the distribution is $s \coloneq \sup_{x \in M} s_x$.  The distribution is \emph{equiregular}  if $\dim \distr^i_x$ is  constant for all $x \in M$ and $i\ge 1$. 
 
\medskip

The generating frame induces an inner product $g_{x}$ on $\distr_{x}$ given by:
\begin{equation}
g_{x}(v,v)\coloneq\min\left\{\sum_{i=1}^{m}u_{i}^{2}:\,  v=\sum_{i=1}^{m}u_{i}\X_{i}(x)\right\},\qquad \forall\, v\in \distr_x.
\end{equation}
A \emph{horizontal curve} is an absolutely continuous (in local charts) map $\gamma : [0,1] \to M$ such that there exists $u\in L^{2}([0,1],
\R^{m})$, called \emph{control}, satisfying
\begin{equation}\label{eq:admissible}
\dot\gamma(t) =  \sum_{i=1}^m u_i(t) \X_i(\gamma(t)), \qquad \mathrm{a.e. }\; t \in [0,1].
\end{equation}
We define the \emph{length} of a horizontal curve
\begin{equation}
\label{lenght}
\ell(\gamma) \coloneq \int_0^1 \sqrt{g_{\gamma(t)}(\dot{\gamma}(t),\dot{\gamma}(t))} \, dt
\end{equation}
and the \emph{sub-Riemannian (or Carnot--Cara\-th\'eodory) distance} between two points $x, y \in M$
\begin{align}\label{CCdis}
d(x,y) \coloneq \inf\{\ell(\gamma): \, \gamma  \mbox{ horizontal}, \gamma(0) = x, \gamma(1) = y\}.    
\end{align}
By the celebrated Rashevskii--Chow theorem, $d$ is a well-defined distance and its induced topology is the same as the manifold topology. In particular, this means that $d$ is continuous. A horizontal path $\gamma: [0,1]\to M$ with constant speed whose length equals the distance between its endpoints is called a \emph{minimizing geodesic}. We say a sub-Riemannian manifold $M$ is \emph{complete} if the pair $(M,d)$ is complete as a metric space. In that case, the infimum in \eqref{CCdis}  is attained.

\medskip

 Instead of $\ell$, it is more convenient to minimize the following {\it energy functional} 
\begin{align} \label{defJ}
\mathcal{J}(\gamma)\coloneq \frac{1}{2} \int_0^1 g_{\gamma(t)}(\dot{\gamma}(t),\dot{\gamma}(t)) \, d t.
\end{align}
Among all horizontal curves with fixed endpoints, and parametrized with constant speed, the minimizers of $\ell$ coincide with the minimizers of $\mathcal{J}$. Furthermore, we have the following relation:
\begin{equation}
\label{CauchProblem}
\frac{1}{2}d^2(x,y) = \inf\{\mathcal{J}(\gamma): \, \gamma \mbox{ horizontal}, \gamma(0) = x, \gamma(1) = y\}.
\end{equation} 
Now fix $x \in M$. Let $\mathcal{U}_x \subset L^2([0,1],\R^m)$ be the set of elements $u$ such that the following equation 
\begin{align}\label{ODEendpoint}
\dot{\gamma}_u(t) = \sum_{i = 1}^m u_i(t) \X_i(\gamma_u(t)), \qquad \gamma_u(0) = x,    
\end{align}
has a solution defined for $t \in [0,1]$. We can define \emph{the end-point map based at $x$}  by 
\begin{equation}
    E_x: \mathcal{U}_x \to M, \qquad u \mapsto \gamma_u(1),
\end{equation}
where $\gamma_u$ is the solution of \eqref{ODEendpoint}. It is well-known that $\mathcal{U}_x$ is open and $E_x$ is a smooth map. It is convenient to define an energy functional on $\mathcal{U}_x$, by letting $J(u) \coloneq \mathcal{J}(\gamma_u)$. Thus, the problem of finding minimizing geodesics amounts to minimize  $J$ among all controls $u \in L^2([0,1],\R^m)$ such that $E_x(u) = y$.  By the method of Lagrange multipliers,  if $u$ is such a minimizing control and $\gamma_u$ is the corresponding minimizing geodesic, there exists a non-trivial pair $(\lambda, \nu)$, such that
\begin{equation}
\label{nu}
\langle \lambda, d E_x (u) \rangle = \nu d J (u), \qquad \lambda \in T_y^* M, \ \nu \in \{0,1\}.
\end{equation}
Here $d$ is the (Fr\'echet) differential and $\la \cdot, \cdot \ra$ denotes the action of covectors on vectors. Given $\nu$ introduced in \eqref{nu}, the minimizing geodesic $\gamma_u$ is called \emph{normal} if $\nu = 1$ and \emph{abnormal} if $\nu = 0$. We remark that a minimizing geodesic could be both normal and abnormal at the same time since the pair $(\lambda, \nu)$ is not necessarily unique. We say that a minimizing geodesic $\gamma : [0,1]\to M$ \emph{contains no non-trivial abnormal segments} if for any $0 \le s_1 < s_2 \le 1$ the restriction $\gamma|_{[s_1,s_2]}$ is not abnormal. A sub-Riemannian structure is \emph{ideal} if it is complete and there are no non-trivial abnormal minimizing geodesics.

\subsection{Measure contraction property}

In the following we assume the underlying  space $(M,d,\mu)$ is a \emph{length space  with negligible cut loci}, i.e.\ a metric measure space with distance function $d$, Borel measure $\mu$, and  for every $x \in M$, there exists a negligible set $\Cut (x)$ and a measurable function $\Phi^x: M \setminus \Cut (x) \times [0,1] \to M$, such that the curve $t \mapsto \Phi^x(y,t)$ is the unique length minimizing geodesic joining $x$ and $y$. On such a space, we define the set of \emph{$s$-intermediate points} by
\[
Z_s(x,A) \coloneq \{\Phi^x(y,s):  y \in A \setminus \Cut (x)\}, \quad \forall \, s \in [0,1], x \in M, A \ \mbox{measurable}.
\]
For metric spaces with negligible cut locus, Ohta's measure contraction property has the following simplified formulation (compare with \cite[Definition 2.1 and Lemma 2.3]{O07}).
\begin{definition}[Measure contraction property]\label{def:MCP}
A length space $(M,d,\mu)$ with negligible cut loci satisfies $\MCP(K,N)$ for some $K \in \R$ and $N \in (1,\infty)$ if for every $x \in M$ and measurable $A \subset M$ with $0<\mu(A)<\infty$  we have
\begin{align}\label{RMCP}
	\mu(Z_s(x,A)) \ge \int_A s \left[ \frac{\mathsf{s}_K(s d(x,y)/\sqrt{N - 1})}{\mathsf{s}_K(d(x,y)/\sqrt{N - 1})} \right]^{N - 1} d\mu(y),
	\qquad\forall \, s\in[0,1],
\end{align}
where we adopt the convention that $0/0 = 1$. Here  the function $\mathsf{s}_K$ is defined by
\[
\mathsf{s}_K(t) \coloneq \begin{cases}
(1/\sqrt{K}) \sin(\sqrt{K} t)  & \text{if $K > 0$}, \\
t  & \text{if $K = 0$}, \\
(1/\sqrt{-K}) \sinh(\sqrt{-K} t)  & \text{if $K < 0$}.
\end{cases}
\]
Similarly, a length space $(M,d,\mu)$ with negligible cut loci satisfies $\MCP_\loc(K,N)$ for some $K \in \R$ and $N \in (1,\infty)$ if for every $x_0 \in M$ there exists a neighbourhood $\cO$ such that for all $x \in \cO$ and all  measurable $A \subset \cO$	with $0<\mu(A)<\infty$, \eqref{RMCP} holds.
\end{definition}

\begin{remark}[On the definition]
In \cite[Definition 2.1]{O07} it is required that, if $K > 0$ and $N>1$, it holds $A \subset B(x,\pi \sqrt{(N - 1)/K})$, where $B(x,r)$ is the open ball centered at $x$ with radius $r$. This requirement is not necessary: in fact, in this case, $M$ has diameter bounded by $\pi \sqrt{(N - 1)/K}$ by \cite[Theorem 4.3]{O07}.   We also note that, in Definition \ref{def:MCP}, the set $Z_s(x,A)$ can be replaced by the slightly larger set obtained by taking the set of all points at time $s \in [0,1]$ of geodesics joining $x$ with $A$ (so, also all geodesics from $x$ to $\Cut(x) \cap A$ are taken into account). The resulting definition is easily seen to be equivalent. Finally, we note that there are different and more general versions of the $\MCP$ in the literature. All of them imply Ohta's version for general metric measure spaces. See for example \cite[\S~10]{BR25}.
\end{remark}

\begin{remark}[The case $N=1$]\label{rmk:N1K>0}
In \cite[Definition 2.1]{O07} Ohta also considers the case $N=1$ and $K\leq 0$. On the other hand, he omits the case $N=1$ and $K>0$, commenting that, in this case, $M$ must be a point; see \cite[Remark 2.2]{O07}. Since giving a unified definition including both cases requires a non-standard modification of the function appearing in the integrand of \eqref{RMCP}, for simplicity we directly consider $N>1$.
\end{remark}

In the following we assume that $M$ is a complete sub-Riemannian manifold, $d$ is the sub-Riemannian distance on $M$, and $\mu$ is a smooth positive measure, i.e., induced by a positive tensor density (or top-dimensional form, if $M$ is orientable), that we denote with the same symbol. The triple $(M,d,\mu)$ is called a \emph{sub-Riemannian metric measure space}. We say that a sub-Riemannian metric measure space $(M,d,\mu)$ satisfies \emph{minimizing Sard property} if for all $x \in M$ we have
\[
\mu(\Abn(x)) = 0, \quad \mbox{where} \quad \Abn(x) \coloneq \{\gamma(1) : \, \gamma \mbox{ abnormal minimizing geodesic starting from } x\}.
\]
For $x\in M$, define the \emph{set of smooth points of $x$} by
\[
\mathcal{S}(x) \coloneq \{y: \, d^2(x, \cdot) \mbox{ is smooth at } y\}.
\]
Under the minimizing Sard property, it follows from \cite[Proposition 2.2]{BR20} that $\mathcal{S}(x)$ has full measure in $M$. As a consequence, $(M,d,\mu)$ is a length space  with negligible cut loci. In particular, if $M$ is ideal, we have $\Abn(x) = \{x\}$ or $\varnothing$ for all $x \in M$ and thus the minimizing Sard property holds automatically. To use the result \cite[Proposition 2.5]{BR20}, we also need to show that $\mathcal{S}(x)$ is \emph{geodesically star-shaped}, namely for any minimizing geodesic $\gamma$ joining $x$ and $y \in \mathcal{S}(x)$, we have $\gamma(s) \in \mathcal{S}(x)$ for all $s \in (0,1]$.

\begin{proposition}\label{pstarshaped}
Assume $M$ is an ideal sub-Riemannian manifold. Then for any $x \in M$ the set $\mathcal{S}(x)$ is geodesically star-shaped.
\end{proposition}

\begin{proof}
Fix $x \in M$ and $y \in \mathcal{S}(x)$. In the case $y = x$ (it may happen for Riemannian points), the minimizing geodesic $\gamma$ joining $x$ and $y$ is a constant curve, and thus $\gamma(s) = x = y \in \mathcal{S}(x)$ for all $s \in (0,1]$. In the opposite case $y \ne x$, if $\gamma$ is a minimizing geodesic joining $x$ and $y$,  any restriction $\gamma|_{[s_1,s_2]}$ ($0 \le s_1 < s_2 \le 1$) is still a minimizing geodesic with positive constant speed, so by assumption it is not abnormal. As a result, $\gamma$ is normal and contains no non-trivial abnormal segments. Then we argue by contradiction. Assume that there is $s \in (0,1)$ such that $\gamma(s) \notin \mathcal{S}(x)$, then \cite[Theorem 11.8]{ABB20} implies two possibilities: either there are two distinct minimizing geodesics joining $x$ to $\gamma(s)$ or $\gamma|_{[0,s]}$ is the unique minimizing geodesic and $\gamma(s)$ is conjugate to $x$ along $\gamma|_{[0,s]}$. However,  \cite[Theorem 8.72]{ABB20} says that in both cases $\gamma$ loses minimality after the time $s$, which is not our case. 
\end{proof}

\subsection{A differential characterization of the \texorpdfstring{$\MCP$}{MCP}}

For a smooth function $f$, we define the \emph{horizontal gradient}, denoted by $\nabla_{\cH} f$,  by the following formula:
\begin{align}\label{defHG}
df_x(v) = g_x(\nabla_{\cH} f(x), v), \qquad \forall\, v \in \cH_x.
\end{align}
Note that if $\X_1, \ldots, \X_m$ is a generating frame, then we have 
\begin{align}\label{soe3}
\nabla_{\cH} f  = \sum_{i = 1}^m  (\X_i f)  \X_i.    
\end{align}
See \cite[Appendix]{RS23} for a proof in the rank-varying case. Given a smooth density $\mu$ and a smooth vector field $\Y$, we recall that the \emph{$\mu$-divergence of $\Y$} is the smooth function $\mathrm{div}_\mu(\Y)$ defined by
\begin{align}\label{defdiv}
  \mathrm{div}_\mu(\Y) \mu \coloneq   \mathcal{L}_\Y \mu,
\end{align}
where $\mathcal{L}_\Y$ is the Lie derivative in the direction of $\Y$. The \emph{sub-Laplacian associated with $\mu$} is given by 
\begin{align}\label{defLap}
\Delta_\mu f \coloneq \mathrm{div}_\mu(\nabla_{\cH} f).
\end{align}
Now we give a slight generalization of \cite[Proposition 2.5]{BR20}. The difference is that we allow the distribution to be rank-varying and we work with general $K$ instead of $0$. For the sake of completeness, we include a sketch of the proof.

\begin{proposition}\label{pchaMCP}
Assume $(M,d,\mu)$ is a sub-Riemannian metric measure space satisfying minimizing Sard property and the sets of smooth points $\mathcal{S}(x)$ are geodesically star-shaped for all $x \in M$. Let $K \in \R$ and $N > 1$. Then $(M,d,\mu)$ satisfies $\MCP(K,N)$ if and only if for all $x \in M$ we have
\begin{align}\label{echaMCP}
    \frac{1}{2} \Delta_\mu d_x^2 \le 1 + \sqrt{N - 1} d_x  \frac{\mathsf{s}_K'(d_x/\sqrt{N - 1})}{\mathsf{s}_K(d_x/\sqrt{N - 1})} \qquad \mbox{on}  \ \cS(x) ,
\end{align}
where we use the notation $d_x^2(\cdot) \coloneq d^2(x,\cdot)$. Furthermore, $(M,d,\mu)$ satisfies $\MCP_\loc(K,N)$ if and only if  for any $x_0 \in M$ there exists a neighbourhood $\cO$  of $x_0$  such that for all $x \in \cO$ the inequality \eqref{echaMCP} holds on $\cS(x) \cap \cO$ instead of $\cS(x)$. 
\end{proposition}
\begin{proof}
Let $x \in M$ be fixed and set $\Y \coloneq  - \frac{1}{2} \nabla_\cH d_x^2$. Now $\Y$ is a well-defined, smooth vector field on $\cS(x)$. Furthermore, it follows from \cite[Proposition 2]{RT05} that on $\cS(x)$ we have
\[
g(\Y, \Y) = \sum_{i = 1}^m \left( \frac{1}{2} \X_i d_x^2 \right)^2 =  d_x^2.
\]
Using an argument similar to the one in the proof of \cite[Proposition 2.5]{BR20}, we find that if $A \subset \cS(x)$, then
\[
Z_s(x,A) = \varphi_t(A) \quad \mbox{with} \quad s = e^{-t}, t \ge 0,
\]
where $\{\varphi_t\}_{t \ge 0}$ is the flow of $\Y$ on $\cS(x)$. Then \cite[Proposition B.1]{BSR18} says that under the relation $s = e^{-t} \in (0, 1]$ we have
\begin{align}
\mu(\varphi_t(A)) & = \int_A \exp\left[ \int_0^t \mathrm{div}_\mu (\Y)(\varphi_\tau(y)) d\tau \right] d\mu(y) = \int_A \exp\left[ - \frac{1}{2} \int_0^t \Delta_\mu d_x^2 (\varphi_\tau(y)) d\tau \right] d\mu(y)  \label{relationA}
\end{align}
while 
\begin{multline}\label{relationB}
\int_A s \left[ \frac{\mathsf{s}_K(s d(x,y)/\sqrt{N - 1})}{\mathsf{s}_K(d(x,y)/\sqrt{N - 1})} \right]^{N - 1} d\mu(y) \\
=  \int_A 
\exp \left[ - \int_0^t \left( 1 + \sqrt{N - 1} d_x(\varphi_\tau(y))  \frac{\mathsf{s}_K'(d_x(\varphi_\tau(y))/\sqrt{N - 1})}{\mathsf{s}_K(d_x(\varphi_\tau(y))/\sqrt{N - 1})} \right) d\tau   \right] d\mu(y).
\end{multline}
Combining \eqref{relationA} and \eqref{relationB}, we can conclude the proof of the first assertion.  For the local result, the ``only if'' part follows from the same argument on a neighbourhood of a fixed point on which \eqref{RMCP} holds. However, for the ``if'' part, given $x_0 \in M$ and the neighbourhood $\cO$ of $x_0$ on which \eqref{echaMCP} holds on $\cS(x) \cap \cO$, let $r > 0$ such that $B(x_0,r) \subset \cO$ and set $\cO' \coloneq B(x_0,r/3)$. Then for any $x \in \cO'$ and subset $A \subset \cS(x) \cap \cO'$, it holds that
\[
Z_s(x,A) = \varphi_t(A) \subset \cS(x) \cap \cO, \qquad \forall \, s = e^{-t} \in (0,1].
\]
Using \eqref{relationA} and \eqref{relationB} again, we conclude that $\MCP_\loc(K,N)$ holds.
\end{proof}

\begin{remark}\label{rmk:echaMCP2}
Let $K \in \R$ and $N > 1$, the $\MCP(K,N)$ or its local version imply that for any $x \in M$ and any sequence $\{x_j\}_{j \in \N} \subset \cS(x)$ such that $x_j \to x$ as $j \to +\infty$, it holds
\begin{equation}\label{echaMCP2}
\limsup_{j \to +\infty} \Delta_\mu d_x^2(x_j) \le 2 N.
\end{equation}
\end{remark}

\section{Construction of the structure}\label{sE}

We first construct a structure satisfying the assumptions of  the following local version of Theorem \ref{tgol}.

\begin{theorem}\label{tloc}
Let $A\in C^\infty(\mathbb{R})$, and consider on $\mathbb{R}^4$, with
coordinates $(x,y,z,w)$, the sub-Riemannian structure given by the following generating frame:
\begin{equation}\label{framecont}
 \X_1 \coloneq \frac{\partial}{\partial x} + (z + A(y)) \frac{\partial}{\partial w},
 \qquad
 \X_2 \coloneq \frac{\partial}{\partial y},
 \qquad
 \X_3 \coloneq \frac{\partial}{\partial z}.
\end{equation}
Assume that the following conditions hold:
\begin{enumerate}[(i)]
\item \label{assumptioni} $A$ is not affine (i.e.\  $A '' \not\equiv 0$) on any non-trivial closed interval of $\R$;
\item \label{assumptionii} there exists a sequence $\{\tau_j\}_{j \in \N} \subset (0,\infty)$ such that 
\[\tau_j \to 0^+, \qquad A(\tau_j) = 0,\] and
\begin{align}\label{conditionB3}
\frac{\left( \int_0^1 A(\tau_j s) ds \right)^2}{\int_0^1 A(\tau_j s)^2 ds} \to 1^{-}, \qquad \mbox{as} \quad j \to +\infty.
\end{align}
\end{enumerate}
Then the sub-Riemannian structure is complete, equiregular,
of step $2$, and ideal. 

Moreover, for every smooth positive measure $\mu$ on $\mathbb{R}^4$,
the sub-Riemannian metric measure space $(\mathbb{R}^4,d,\mu)$ does not satisfy $\MCP_{\mathrm{loc}}(K,N)$ for any $K\in\mathbb{R}$ and any $N \in (1,\infty)$. In fact, $\MCP_{\mathrm{loc}}(K,N)$ fails in every
neighbourhood of the origin.
\end{theorem}

We arrange this section as follows: in Subsection \ref{ssnc} we prove Theorem \ref{tloc}; in Subsection \ref{sscA} we construct the key function $A$; in Subsection \ref{ssce}, by taking a  quotient with respect to a group action, we compactify the structure obtained in Subsection \ref{ssnc} completing the proof of Theorem \ref{tgol}.

\subsection{Proof of  Theorem \ref{tloc}}\label{ssnc}

The vector field $\X_3$ has the only purpose of making the structure equiregular and of step $2$, in fact it holds
\[
 [\X_3,\X_1]=\frac{\partial}{\partial w}.
\]

Fix $\xi_1 = (x_1, y_1, z_1, w_1) \in \R^4$.  For a control $u\in L^2([0,1],\R^3)$, the corresponding horizontal curve  $\gamma_u: [0,1] \to \R^4$ is the solution  of \eqref{ODEendpoint}.
So if we write $\gamma_u(t) = (x_u(t),y_u(t),z_u(t),w_u(t))$, we obtain
\begin{align}\label{expgammato}
\begin{cases}
\dot{x}_u(t) = u_1(t), \\
\dot{y}_u(t) = u_2(t), \\
\dot{z}_u(t) = u_3(t), \\
\dot{w}_u(t) = (z_u(t) + A(y_u(t))) u_1(t) .
\end{cases}
\end{align}
Consequently, it holds
\begin{align}\label{expgammato2}
\begin{cases}
{x}_u(1) = \displaystyle  \int_0^1 u_1(s) ds + x_1, \\
{y}_u(1) = \displaystyle  \int_0^1 u_2(s) ds + y_1, \\
{z}_u(1) = \displaystyle  \int_0^1 u_3(s) ds + z_1, \\
{w}_u(1) =  \displaystyle \int_0^1 \left[\int_0^t u_3(s) ds + z_1 + A\left(\int_0^t u_2(s) ds + y_1\right)\right] u_1(t) dt + w_1.
\end{cases}.
\end{align}

\begin{lemma}\label{lemcaldis}
Consider the sub-Riemannian structure determined by \eqref{framecont}. For any two points $\xi_1 = (x_1,y_1,z_1,w_1), \xi_2 = (x_2, y_2, z_2, w_2) \in  \R^4$ with $x_1 = x_2$ and $w_1 = w_2$, we have $d^2(\xi_1, \xi_2) = |y_1 - y_2|^2 + |z_1 - z_2|^2$ and the only minimizing geodesics are straight segments between such endpoints, parametrized with constant speed.
\end{lemma}
\begin{proof}
 We note that the projection of $\gamma_u = (x_u,y_u,z_u,w_u)$ on the first three variables follows an Euclidean dynamics on $\R^3$. More precisely, if $\t{\gamma}_u \coloneq (x_u,y_u,z_u)$, then such curve satisfies
\[
\dot{\t{\gamma}}_u = \sum_{i = 1}^3 u_i(t) \Y_i(\t{\gamma}_u(t)), \qquad \mbox{for a.e. }  t \in [0,1], \qquad \t{\gamma}_u(0) = (x_1,y_1,z_1),
\]
where 
\[
\Y_1 \coloneq \frac{\partial}{\partial x}, \qquad  \Y_2 \coloneq \frac{\partial}{\partial y}, \qquad \Y_3 \coloneq \frac{\partial}{\partial z}.
\]
It follows from \eqref{CauchProblem} that 
\begin{align}
d^2(\xi_1,\xi_2) & = \inf \{\|u\|^2_{L^2([0,1],\R^3)}:  \, \gamma_u(0) = \xi_1,\; \gamma_u(1) = \xi_2\} \\
& \ge  \inf \{\|u\|^2_{L^2([0,1],\R^3)}:   \, \t{\gamma}_u(0) = (x_1,y_1,z_1),\; \t{\gamma}_u(1) = (x_2,y_2,z_2)\}\\
& = |y_1 - y_2|^2 + |z_1 - z_2|^2,
\end{align}
 where we used that $x_1=x_2$. To prove the opposite inequality, by the special choice of control $u^*_{\xi_1,\xi_2}(t) \coloneq (0,y_2 - y_1, z_2 - z_1)$, we know $\gamma_{u^*_{\xi_1,\xi_2}}(1) = \xi_2$ by \eqref{expgammato2} and thus 
\[
d^2(\xi_1,\xi_2) \le |y_1 - y_2|^2 + |z_1 - z_2|^2.
\]
Since on $\R^3$ the only minimizing geodesics are straight lines with constant speed, we proved the last assertion.
\end{proof}

\begin{lemma}\label{lcompl}
For every $A\in C^\infty(\R)$, the sub-Riemannian structure given by \eqref{framecont} is complete.
\end{lemma}

\begin{proof}
In the  proof we will use $\l{B}(\xi,r)$ to denote the closed sub-Riemannian ball centered at $\xi$ with radius $r$ (recall that by the Rashevskii--Chow theorem the metric topology and the manifold topology coincide). Fix $\xi_1 = (x_1,y_1,z_1,w_1) \in \R^4$ and $r > 0$. For any $\xi_2 = (x_2,y_2,z_2,w_2) \in \l{B}(\xi_1,r)$, it follows that there exists $u \in L^2([0,1],\R^3)$ such that $\gamma_u(1) = \xi_2$ and $\|u\|_{L^1([0,1],\R^3)} < 2r$ . As a consequence, it follows from \eqref{expgammato2} that 
\[
|x_1 - x_2| \le 2r, \quad |y_1 - y_2| \le 2r, \quad |z_1 - z_2| \le 2r, \quad |w_1 - w_2| \le 2r \left( |z_1| + 2r + \max_{y \in [y_1 - 2r,y_1 + 2r]} |A(y)| \right),
\]
which implies that $\l{B}(\xi_1,r)$ is contained in some bounded set of $\R^4$ and thus it is compact. Then the completeness of $(\R^4,d)$ follows from \cite[Proposition 3.47]{ABB20}.    
\end{proof}

\begin{proposition}\label{pideal}
If $A \in C^\infty(\R)$ is not affine on any non-trivial closed interval,  then $\R^4$ equipped with the sub-Riemannian structure given by \eqref{framecont} is ideal.  
\end{proposition}

\begin{proof}
The completeness follows from Lemma \ref{lcompl}. Now we prove that there are no non-trivial abnormal minimizers. Assume $\gamma_u = (x_u,y_u,z_u,w_u)$ is an abnormal minimizing geodesic joining $\xi_1 = (x_1,y_1,z_1,w_1)$ and $\xi_2 = (x_2,y_2,z_2,w_2)$. By definition there exists $\lambda  = (\lambda_1, \lambda_2, \lambda_3,\lambda_4) \in \R^4 \setminus \{0\}$ such that
\begin{align}\label{abnormal}
\la \lambda, dE_{\xi_1}(u) \ra = 0.    
\end{align}
Using \eqref{expgammato2}, we can calculate the differential of the end-point map $E_{\xi_1}$ at $u$:
\begin{align}\label{expdiffend}
d E_{\xi_1}(u) v = \begin{pmatrix}
\int_0^1 v_1(s) ds \\
\int_0^1 v_2(s) ds \\
\int_0^1 v_3(s) ds \\
\int_0^1 [z_u(s) + A(y_u(s))] v_1(s) ds
+ \int_0^1 u_1(t) \left[ \int_0^t v_3(s) ds + A'(y_u(t)) \int_0^t v_2(s) ds \right] dt
\end{pmatrix}
\end{align}
for all $v \in L^2([0,1],\R^3)$. Combining \eqref{abnormal} with \eqref{expdiffend}, we obtain
\begin{align}\label{expdiffend2}
\begin{cases}
   \lambda_1 + \lambda_4 (z_u(s) + A(y_u(s)) ) = 0, \\
    \lambda_2 + \lambda_4 \int_s^1 A'(y_u(t)) u_1(t) dt = 0, \\
    \lambda_3 + \lambda_4 \int_s^1 u_1(t) dt  = 0,
\end{cases} \qquad \qquad \mbox{for a.e.} \ s \in [0,1].
\end{align}
To be more precise, to obtain the third equation in \eqref{expdiffend2}, we choose $v = (0,0,v_3)$ with $v_3 \in L^2([0,1])$. Then  \eqref{abnormal}, together with \eqref{expdiffend}, gives
\begin{align*}
 \int_0^1 \left( \lambda_3 + \lambda_4 \int_s^1 u_1(t) dt \right) v_3(s) ds 
 =    \lambda_3 \int_0^1 v_3(s) ds + \lambda_4 \int_0^1 \int_0^t u_1(t)  v_3(s) ds dt = 0.
\end{align*}
Since $v_3$ is arbitrary, we obtain the third equation in \eqref{expdiffend2}. The other two equations follow by a similar argument. Now from \eqref{expdiffend2} we know that $\lambda_4 \ne 0$. Otherwise \eqref{expdiffend2} yields $\lambda_1 = \lambda_2 = \lambda_3 = 0$ and thus $\lambda = 0$, which leads to a contradiction. Dividing both sides of \eqref{expdiffend2} by $\lambda_4$, we know from the third equation of \eqref{expdiffend2} that $u_1 = 0$. Inserting this into \eqref{expgammato2} we obtain that $x_2 = x_1$ and $w_2 = w_1$. It follows from Lemma \ref{lemcaldis} that $u = u^*_{\xi_1,\xi_2} = (0,y_2 - y_1, z_2 - z_1)$ since $\gamma_u$ is minimizing and 
\[
y_u(s) = y_1 + s(y_2 - y_1), \qquad z_u(s) = z_1 + s (z_2 - z_1).
\]
Inserting these into the first equation of \eqref{expdiffend2} we get
\begin{align}\label{expdiffend3}
\lambda_1 + \lambda_4 (z_1 + s (z_2 - z_1) + A(y_1 + s(y_2 - y_1))) = 0, \qquad \mbox{for a.e.} \ s \in [0,1].    
\end{align}
There are two possibilities: either $y_2 = y_1$ and \eqref{expdiffend3} implies $z_2 = z_1$, or $y_2 \ne y_1$ and \eqref{expdiffend3} implies $A$ is an affine function on some non-trivial closed interval of $\R$. Our assumption rules out the second possibility so we must have $\xi_1 = \xi_2$ and $\gamma_u$ is a constant curve. In conclusion, we proved that whenever we have an abnormal minimizing geodesic, it must be the trivial one.
\end{proof}

\begin{corollary}\label{csmooth}
If $A  \in C^\infty(\R)$ is not affine on any non-trivial closed interval of $\R$, then for the sub-Riemannian structure determined by \eqref{framecont}, we have $\{0\} \times (\R^2 \setminus\{0\}) \times \{0\} \subset \mathcal{S}(0)$. 
\end{corollary}

\begin{proof}
We use \cite[Theorem 11.8]{ABB20}. Set $\xi_1 = 0$ and $\xi_2 = (0,y_2,z_2,0) \ne 0$. Lemma \ref{lemcaldis} implies that the straight line $[0,1]\ni t \mapsto t \xi_2$ is the unique minimizing geodesic between these points. From our assumption on the function $A$, together with Proposition \ref{pideal}, we know the minimizing geodesic contains no non-trivial abnormal segments, so a fortiori it is strictly normal (i.e.\ it is not abnormal). So it is sufficient to prove that  $\xi_2$ is not conjugate to $\xi_1$ along the minimizing geodesic. We argue by contradiction. Assume on the contrary that $\xi_2$ is conjugate to $\xi_1$ along the minimizing geodesic, \cite[Theorem 8.72]{ABB20} says that the minimizing geodesic loses minimality after the time $1$, which is not the case  for this straight line.
\end{proof}

In the following we use $\mathscr{L}$ to denote the Lebesgue measure  on $\R^4$. For any smooth positive measure $\mu$, there is a smooth function $V$ such that $\mu =e^{-V}\mathscr{L}$.  Then the sub-Laplacian associated with $\mu$ in this case (recalling \eqref{defLap}) is given by
\begin{align}\label{defsubLap}
\Delta_\mu = \X_1^2 + \X_2^2 + \X_3^2 - (\X_1 V)  \X_1 -  (\X_2 V)  \X_2 - (\X_3 V) \X_3.
\end{align}
In particular, we have 
\begin{align}\label{defsubLap2}
\Delta_\mathscr{L} = \X_1^2 + \X_2^2 + \X_3^2.
\end{align}
 Our aim is to prove that $\Delta_\mu d_0^2$ is unbounded near $0$, to prove, via Proposition \ref{pchaMCP} and Remark \ref{rmk:echaMCP2}, that $\MCP_{\mathrm{loc}}$ must fail. To this purpose, we compute $\Delta_\mu d_0^2$ at points $(0,\tau,0,0)$ with $\tau > 0$. 

\begin{lemma}\label{lemcalX12}
For the sub-Riemannian structure determined by \eqref{framecont}, assume that $d_0^2$ is smooth near $(0,\tau,0,0)$ with $\tau > 0$, and $A(\tau) = 0$ but $A \not\equiv 0$  on $[0,\tau]$. Then we have 
\begin{equation}\label{eq:denominator}
\X_1^2 d^2_0 ((0,\tau,0,0)) = 2 \, \frac{\int_0^1 A(\tau s)^2 ds}{\int_0^1 A(\tau s)^2 ds - \left( \int_0^1 A(\tau s) ds \right)^2},
\end{equation}
and 
\begin{align}\label{calX123}
\X_2^2 d^2_0 ((0,\tau,0,0)) = \X_3^2 d^2_0 ((0,\tau,0,0)) = 2.
\end{align}
 Note that by our assumptions and Cauchy--Schwarz inequality, the denominator in \eqref{eq:denominator} is positive.
\end{lemma}

\begin{proof}
From Lemma \ref{lemcaldis} it holds
\[
d^2_0((0,\tau + \varepsilon, 0, 0)) = |\tau + \varepsilon|^2, \qquad d^2_0((0,\tau, \varepsilon, 0)) = \tau^2 + \varepsilon^2,
\]
which imply \eqref{calX123}. We are left to compute $\X_1^2 d^2_0 ((0,\tau,0,0))$.

Since $d_0^2$ is smooth near $(0,\tau,0,0)$ and $A(\tau) = 0$, by inspection of \eqref{framecont} we have 
\[
\X_1^2 d^2_0 ((0,\tau,0,0)) =   \partial_x^2 d^2_0 ((0,\tau,0,0)) .
\]
By Taylor expansion, it is  then sufficient to consider the following limit of second-order difference:
\begin{align}\label{limit2nd}
\X_1^2 d^2_0 ((0,\tau,0,0)) = \lim_{\varepsilon \to 0^+} \frac{d^2_0 ((\varepsilon,\tau,0,0)) + d^2_0 ((-\varepsilon,\tau,0,0)) - 2 d^2_0 ((0,\tau,0,0))}{\varepsilon^2}.
\end{align}
Observing that if $\gamma_u$ is a curve that joins $\xi_1=0$ to $\xi_2=(\varepsilon,\tau,0,0)$, then $\gamma_{\t{u}}$ with $\t{u} = (-u_1, u_2, u_3)$ is a curve that joins $0$ to $(-\varepsilon,\tau,0,0)$. So we must have $d^2_0 ((\varepsilon,\tau,0,0)) = d^2_0 ((-\varepsilon,\tau,0,0))$. Inserting this into \eqref{limit2nd}, we obtain
\begin{align}\label{calX12}
\X_1^2 d^2_0 ((0,\tau,0,0)) = 2 \lim_{\varepsilon \to 0^+} \frac{d^2_0 ((\varepsilon,\tau,0,0))  -  d^2_0 ((0,\tau,0,0))}{\varepsilon^2}.
\end{align}
For any fixed $\varepsilon > 0$, from \eqref{CauchProblem} (together with the completeness obtained from Lemma \ref{lcompl}) there exists a $u^\varepsilon \in L^2([0,1],\R^3)$ such that $\gamma_{u^\varepsilon}(1) = (\varepsilon,\tau,0,0)$ and $d^2_0 ((\varepsilon,\tau,0,0)) = \|u^\varepsilon\|^2_{L^2([0,1],\R^3)}$. To be more precise, we have from \eqref{expgammato2} that
\begin{gather}\label{boundary}
\int_0^1 u^\varepsilon_1 (s) ds = \varepsilon,\qquad 
\int_0^1 u^\varepsilon_2 (s) ds = \tau,\qquad 
\int_0^1 u^\varepsilon_3 (s) ds = 0, \\
\label{boundary2}
\int_0^1 \int_0^t u^\varepsilon_3(s) u_1^\varepsilon(t) ds dt  + \int_0^1  A\left(\int_0^t u_2^\varepsilon(s) ds\right) u_1^\varepsilon(t) dt = 0.
\end{gather}
Notice that $d^2_0 ((0,\tau,0,0)) = \tau^2$ by Lemma \ref{lemcaldis}. Combining this with the second equation of \eqref{boundary}, we get
\[
\frac{d^2_0 ((\varepsilon,\tau,0,0))  -  d^2_0 ((0,\tau,0,0))}{\varepsilon^2} = \varepsilon^{-2} \left( \int_0^1 (u^\varepsilon_1 (s))^2 ds + \int_0^1 (u^\varepsilon_2 (s) - \tau)^2 ds + \int_0^1 (u^\varepsilon_3 (s))^2 ds \right).
\] 
Since the left-hand side converges as $\varepsilon \to 0^+$, we note that $u^\varepsilon_2 \to \tau$, $u^\varepsilon_3 \to 0$ in $L^2([0,1])$. Furthermore, $\varepsilon^{-1}u_1^\varepsilon$ is bounded in $L^2([0,1])$. By choosing a  suitable subsequence $\{\varepsilon_j\}_{j \in \N}$, the sequence $\varepsilon_j^{-1} u_1^{\varepsilon_j}$ converges weakly to some $\omega \in L^2([0,1])$  as $j\to +\infty$. We conclude that 
\begin{align*}
\|\omega\|_{L^2([0,1])}^2 & \le \liminf_{j \to +\infty} \| \varepsilon_j^{-1} u_1^{\varepsilon_j} \|_{L^2([0,1])}^2 \\
& \le \liminf_{j \to +\infty} \varepsilon_j^{-2} \left( \int_0^1 (u^{\varepsilon_j}_1 (s))^2 ds + \int_0^1 (u^{\varepsilon_j}_2 (s) - \tau)^2 ds + \int_0^1 (u^{\varepsilon_j}_3 (s))^2 ds \right) \\
& = \lim_{j \to +\infty} \frac{d^2_0 ((\varepsilon_j,\tau,0,0))  -  d^2_0 ((0,\tau,0,0))}{\varepsilon_j^2} = \frac{1}{2} \X_1^2 d^2_0 ((0,\tau,0,0)),
\end{align*}
where we used \eqref{calX12} in the last equality. Note that the first equation of \eqref{boundary} and \eqref{boundary2}  yield
\[
\int_0^1 \omega(s) ds = 1, \qquad \int_0^1 \omega(s) A(\tau s) ds = 0.
\]
So we have proved 
\[
\frac{1}{2} \X_1^2 d^2_0 ((0,\tau,0,0)) \ge \inf\left\{\|\omega\|_{L^2([0,1])}^2: \, \omega \in L^2([0,1]), \int_0^1 \omega(s) ds = 1,  \int_0^1 \omega(s) A(\tau s) ds = 0\right\}.
\]
In fact, the opposite inequality is also true since for any 
\[
\omega \in L^2([0,1]), \quad \mbox{with} \quad  \int_0^1 \omega(s) ds = 1, \quad  \int_0^1 \omega(s) A(\tau s) ds = 0,
\]
we can set $\h{u}^\varepsilon_\omega \coloneq (\varepsilon \omega, \tau, 0)$ and it is not hard to show that $\gamma_{\h{u}^\varepsilon_\omega}$ is a curve joining $0$ to $(\varepsilon, \tau, 0, 0)$. As a consequence, we obtain
\[
d^2_0 ((\varepsilon,\tau,0,0)) \le \varepsilon^2 \|\omega\|_{L^2([0,1])}^2 + \tau^2,
\]
which implies the opposite inequality by \eqref{calX12} and by taking infimum with respect to all such $\omega$. 

\medskip

It remains to  compute
\begin{align}\label{calX122}
\inf\left\{\|\omega\|_{L^2([0,1])}^2: \, \int_0^1 \omega(s) ds = 1,  \int_0^1 \omega(s) A(\tau s) ds = 0\right\}.
\end{align}
Define smooth maps $\varphi: L^2([0,1]) \to \R$ and $F: L^2([0,1]) \to \R^2$ by
\[
\varphi(\omega) \coloneq \|\omega\|_{L^2([0,1])}^2, \qquad F(\omega) \coloneq \begin{pmatrix}
\int_0^1 \omega(s) ds \\
\int_0^1 \omega(s) A(\tau s) ds
\end{pmatrix}.
\]
So  that \eqref{calX122} is precisely $\inf \varphi|_{F^{-1}((1,0))}$.  This is a minimization problem with a convex functional and linear constraint, so that it has a unique solution; we denote it by $\omega_*$. By the method of Lagrange multipliers there is $\nu \coloneq (\nu_1,\nu_2,\nu_3) \ne 0$ such that
\[
\nu_1 d \varphi(\omega_*) + \langle (\nu_2, \nu_3), d F(\omega_*) \rangle = 0,
\]
which is exactly 
\[
 \int_0^1 (2 \nu_1 \omega_*(s) + \nu_2 + \nu_3 A(\tau s))  v(s) ds  = 0, \qquad \forall \, v \in L^2([0,1]).
\]
This implies
\[
2 \nu_1 \omega_*(s) + \nu_2 + \nu_3 A(\tau s) = 0, \qquad \mbox{for a.e.} \ s \in [0,1].
\]
If $\nu_1 =0$, then $[0,1]\ni s \mapsto \nu_3  A(\tau s)$ must be constant. Since by construction $A(\tau)=0$ and $A\not\equiv 0$ on $[0,\tau]$ this yields $\nu=0$, which is not possible. As a consequence, we assume $\nu_1 =-1/2$ and get
\[
 \omega_*(s) =  \nu_2 + \nu_3 A(\tau s).
\]
Since $F(\omega_*) = (1,0)$, we  obtain
\begin{equation}
    \nu_2 = \frac{\int_0^1 A(\tau s)^2ds}{\int_0^1 A(\tau s)^2ds - \left( \int_0^1 A(\tau s) ds \right)^2}, \qquad \nu_3 = -  \frac{\int_0^1 A(\tau s)ds}{\int_0^1 A(\tau s)^2ds - \left( \int_0^1 A(\tau s) ds \right)^2}.
\end{equation}
Finally we can calculate
\[
\varphi(\omega_*) = \frac{\int_0^1 A(\tau s)^2 ds}{\int_0^1 A(\tau s)^2ds - \left( \int_0^1 A(\tau s) ds \right)^2},
\]
and we conclude the proof.
\end{proof}

\begin{proof}[Proof of Theorem \ref{tloc}]
It follows from Proposition \ref{pideal} that the sub-Riemannian structure given by \eqref{framecont} is step $2$, equiregular, and ideal. Let $\{\tau_j\}_{j \in \N}$ be the sequence of assumption \eqref{assumptionii} on the function $A$. By Corollary \ref{csmooth} we have $\{(0,\tau_j,0,0)\}_{j \in \N} \subset \mathcal{S}(0)$. As a consequence, Lemma \ref{lemcalX12}, together with \eqref{conditionB3}, gives (recalling \eqref{defsubLap2})
\[
\Delta_\mathscr{L} d_0^2 ((0,\tau_j,0,0)) = 4 + \frac{2}{1 - \frac{\left( \int_0^1 A(\tau_j s) ds \right)^2}{\int_0^1 A(\tau_j s)^2 ds}} \to +\infty, \qquad \mbox{as} \quad j \to +\infty.
\]
Now let $\mu$ be a smooth positive measure.  To obtain an analogous limit for $\Delta_\mu d_0^2 ((0,\tau_j,0,0))$ we use the
horizontal Eikonal equation (see e.g.\ \cite[Proposition 5.7]{FR10}), which implies
\[
(\X_1 d_0((0,\tau_j,0,0)))^2 + (\X_2 d_0((0,\tau_j,0,0)))^2 + (\X_3 d_0((0,\tau_j,0,0)))^2 = 1.
\]
Thus, the term $|(\X_1 V)  \X_1 d_0^2  + (\X_2 V)   \X_2 d_0^2 + (\X_3 V) \X_3 d_0^2|$ appearing in \eqref{defsubLap} is bounded on the sequence $\{(0,\tau_j,0,0)\}_{j \in \N}$. In conclusion, we obtain the limit
\begin{align}\label{keylim}
\Delta_\mu d_0^2 ((0,\tau_j,0,0))   \to +\infty, \qquad \mbox{as} \quad j \to +\infty.
\end{align}
Finally, thanks to Proposition \ref{pstarshaped}, we can apply  the characterization of $\MCP$ of Proposition \ref{pchaMCP} and, more precisely, Remark \ref{rmk:echaMCP2}, to conclude that \eqref{keylim} contradicts $\MCP_{\mathrm{loc}}(K,N)$ for any $K \in \R$ and $N > 1$.
\end{proof}

\subsection{Construction of the function \texorpdfstring{$A$}{A}}\label{sscA}

We now prove the existence of a function satisfying the assumptions \eqref{assumptioni}--\eqref{assumptionii} of Theorem \ref{tloc}. The additional assumption below is only needed for the subsequent compactification.

\begin{proposition}\label{propconA}
There is a function $A \in C^\infty(\R)$ such that
\begin{enumerate}[(i)]
\item $A$ is not affine on any non-trivial closed interval of $\R$;
\item there exists a sequence $\{\tau_j\}_{j \in \N} \subset (0,1]$ such that $\tau_j \to 0^+$, $A(\tau_j) = 0$, and
\begin{align}\label{conditionB2}
\frac{\left( \int_0^1 A(\tau_j s) ds \right)^2}{\int_0^1 A(\tau_j s)^2 ds} \to 1^{-}, \qquad \mbox{as} \quad j \to +\infty;
\end{align}
\item $A$ is periodic with period $1$.
\end{enumerate}
\end{proposition}

\begin{proof}
\begin{figure}[htp]
	\centering
    \includegraphics[scale = 0.25]{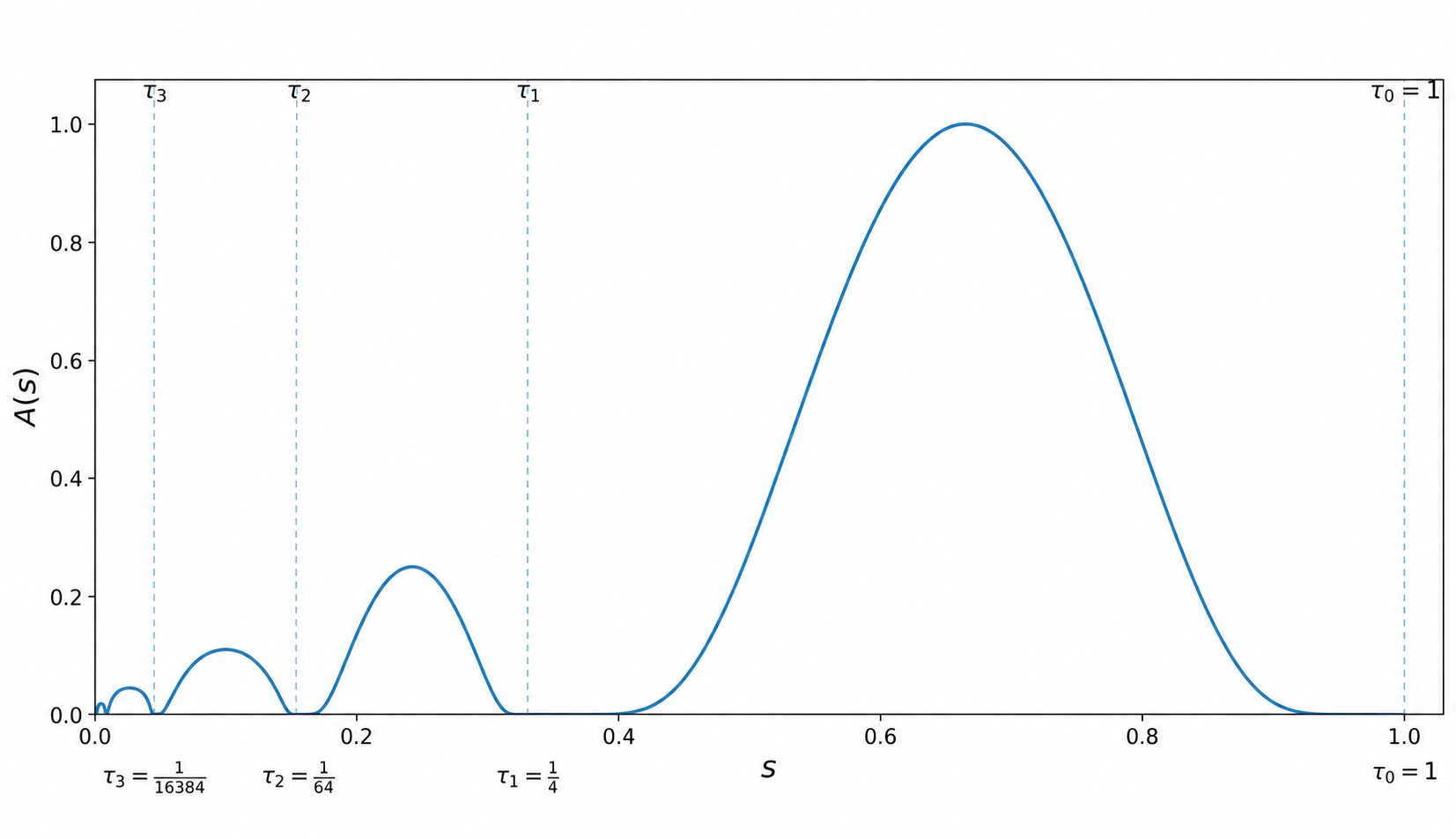}
	\caption[image]{Plot of the function $A$ on $[0,1]$.}\label{FigA}
\end{figure}

 Set $\tau_0 = 1$, define $\{\tau_j\}_{j \in \N}$ recursively by 
\begin{align}\label{deftauj}
\tau_{j+1}=\frac{\tau_j^2}{4}, \qquad \forall \, j \in \N.
\end{align}
Then for any $\delta >0$, we define the auxiliary function $\psi_\delta: \R \to \R$ by
\begin{align}\label{defpsid}
\psi_\delta(s)  \coloneq
\begin{cases}
 \exp\left(\delta \left( 4 -\frac{1}{s(1-s)}\right) \right), \qquad 
      & 0 < s < 1,\\
  0,  & s \notin (0,1).
 \end{cases}
\end{align}
Notice that  $\psi_\delta \in C^\infty(\R)$,
\[
0 \le \psi_\delta \le 1, \qquad \psi_\delta\left( \frac{1}{2} \right) = 1,
\]
and 
\[
\psi_\delta \to 1, \qquad \mbox{as} \quad \delta \to 0^+,
\]
pointwise in $(0,1)$ and thus as a function in $L^2([0,1])$ by the dominated convergence theorem. Now we give the construction of our function $A$ first on $[0,1]$:
\begin{equation}\label{defB}
 A(s) \coloneq \begin{cases}
a_j \psi_{2^{-j}} \left(\frac{s - \tau_{j+1}}{\tau_j - \tau_{j + 1}}\right), \qquad & s \in (\tau_{j + 1}, \tau_j), j \in \N, \\
0, \qquad & s = \tau_j, j \in \N \quad \mbox{or} \quad s = 0,
\end{cases}
\end{equation}
where $\{a_j\}_{j \in \N} \subset (0,1]$ is a decreasing sequence to guarantee the smoothness of the function $A$ at $0$ (by ensuring that every derivative vanishes at $0$). For example, we can choose $a_0 = 1$ and define $\{a_j\}_{j \in \N}$ recursively by the following process: having chosen $a_j$, we let
\[
a_{j + 1} \coloneq \min\left\{ a_j, \frac{1}{j + 1} \left( 1 + \sum_{\ell = 0}^{j + 1} (\tau_{j + 1} - \tau_{j + 2})^{-\ell}  \|\psi_{2^{-j - 1}}^{(\ell)}\|_{L^\infty (\R)}   \right)^{-1} \right\}, \qquad \forall \, j \in \N.
\]
With this choice of $\{a_j\}_{j \in \N}$ we can periodically extend $A$ to a smooth periodic function on $\R$, periodic with period $1$ (since all derivatives vanish at both $0$ and $1$). By construction $A$ is not affine on any non-trivial closed interval of $\R$. 

We are left to prove \eqref{conditionB2}. In fact, it is sufficient to prove $A(\tau_j \cdot)/a_j \to 1$ as $j \to +\infty$ in $L^2([0,1])$. Define $r_j \coloneq \frac{\tau_{j + 1}}{\tau_j} = \frac{\tau_j}{4} \to 0^+$ by \eqref{deftauj}. We split the following integral into two parts:
\[
\int_0^1 \left| \frac{A(\tau_j s)}{a_j} - 1\right|^2 ds
= \int_0^{r_j} \left| \frac{A(\tau_j s)}{a_j} - 1\right|^2 ds + \int_{r_j}^1 \left| \frac{A(\tau_j s)}{a_j} - 1\right|^2 ds.
\]
Since by construction $A(\tau_j s) \leq a_j$ for $s \in [0,1]$, for the first integral we have the estimate:
\[
\int_0^{r_j} \left| \frac{A(\tau_j s)}{a_j} - 1\right|^2 ds \le 4 r_j \to 0, \quad \mbox{as} \quad j \to +\infty.
\]
For the second integral by a change of variables we have 
\[
\int_{r_j}^1 \left| \frac{A(\tau_j s)}{a_j} - 1\right|^2 ds = (1 - r_j) \int_0^1 |\psi_{2^{-j}}(s) - 1|^2 ds \to 0 \quad \mbox{as} \quad j \to +\infty.
\]
This concludes the proof. 
\end{proof}

\begin{remark}[A technical comment]\label{rnondege}
The function $A$ given in Proposition \ref{propconA} itself is not a counterexample to the validity of \cite[Lemma 2.11]{BR20} in the smooth setting. However, with such $A$ we can construct a counterexample: define $h(t,\kappa) \coloneq A(\kappa t)$, which is a smooth function on $[0,1] \times [0,1]$; then for any $\tau \in (0,1]$ and $j \in \N$ such that $\tau_j \le \tau$, after a change of variables \eqref{conditionB2} becomes
\[
\frac{\left( \int_0^{\tau} h(t,\kappa_j) dt \right)^2}{\tau \int_0^{\tau} h(t,\kappa_j)^2 dt}
= \frac{ \left( \int_0^{\tau} A(\kappa_j t) dt \right)^2}{\tau \int_0^{\tau} A(\kappa_j t)^2 dt} \to 1^{-}, \qquad \mbox{as} \quad j \to +\infty,
\]
where $\kappa_j := \tau_j / \tau$. Thus, there is no $\nu \in (0,1)$ such that 
\[
\left( \int_0^{\tau} h(t,\kappa) dt \right)^2 \le \nu \tau \int_0^{\tau} h(t,\kappa)^2 dt, \qquad \forall \, \kappa \in [0,1].
\]
In fact, such function $h$ is closely related with the function $\t{s}$ or the function $h$ in the proof of \cite[Lemma 2.10]{BR20} for the sub-Riemannian structure given by \eqref{framecont}.
On the other hand, using the notation in \cite{BR20}, if in addition we have a number $N \ge 1$ such that  $\sum_{j = 1}^N |\partial_t^j h(0,\kappa)| \ne 0$ for all $\kappa \in \mathcal{K}$, it is not hard to show that an analogue of \cite[Lemma 2.11]{BR20} still holds. A result of this kind can be used to prove that every compact sub-Riemannian manifold with fat distribution admits $\MCP(0,N)$ for some finite $N$. This would be a simpler proof of the one in \cite[Corollary 9.24]{BMR26}, that is based on the so-called canonical frame.
\end{remark}

\subsection{Proof of Theorem \ref{tgol}}\label{ssce}

 In this section we compactify the construction of Theorem \ref{tloc}.
\begin{proof}
Let $M = \R^4$ be the example obtained in Theorem \ref{tloc} with the function $A$ satisfying all the properties in Proposition \ref{propconA}. 
We give $M = \R^4$ a  Lie group structure by the following group law:
\begin{align}\label{groupa}
    (x_1,y_1,z_1,w_1) \cdot  (x_2,y_2,z_2,w_2) \coloneq  (x_1 + x_2, y_1 + y_2, z_1 + z_2, w_1 + w_2 + x_1 z_2).
\end{align}
With this group law $\Z^4$ forms a closed subgroup and we write a point in $\Z^4$  as $\sigma$. Now we let $\t{M} \coloneq M / \Z^4$ and let $\pi: M \to \t{M}$ be the quotient map. Since $A$ is periodic, we obtain
\[
dR_{\sigma} (\xi) \X_i(\xi) =  \X_i(\xi \cdot \sigma), \qquad \forall \, i = 1,2,3, \xi \in M, \sigma \in \Z^4,
\]
where $R_{\sigma}(\xi) \coloneq \xi  \cdot \sigma $, and we obtain that 
\[
\t{\X}_1 \coloneq d\pi \, \X_1 , \quad \t{\X}_2 \coloneq d\pi \, \X_2, \quad \t{\X}_3 \coloneq d\pi \, \X_3,
\]
give a well-defined generating frame for $\t{M}$. By construction, the sub-Riemannian structure has step $2$ and is equiregular. Notice that $\t{M}$ is compact, which implies that it is also complete. If we use  $\t{E}_{\t{\xi}}$ and $\t{d}$ to denote the  end-point map based at $\t{\xi}$ and the sub-Riemannian distance on $\t{M}$ respectively, it is not hard to show the following:
\begin{align}
 \label{relend}
  \t{E}_{\pi(\xi)}(u) &= \pi(E_\xi (u)), \qquad \forall \, \xi \in M, u \in \mathcal{U}_\xi, \\
 \label{reldis}
 \t{d}(\pi(\xi_1), \pi(\xi_2)) &= \min \{d(\xi_1, \xi_2 \cdot \sigma) : \, \sigma \in \Z^4\},  \qquad \forall \, \xi_1, \xi_2 \in M. 
\end{align}
With the help of \eqref{relend} and \eqref{reldis}, we can show that $\t{M}$ is also ideal. Furthermore, thanks to \eqref{reldis} again, in a sufficiently small neighbourhood of $0$ in $M$,  we have 
\[
\t{d}^2_{\pi(0)} (\pi(\xi)) = \t{d}^2(\pi(0),\pi(\xi)) = d^2(0,\xi) = d_0^2(\xi).
\]
We obtain that, when $\tau > 0$ and $\tau$ small, $\pi((0,\tau,0,0))$ belongs to the set of smooth points of $\pi(0)$ and
\begin{align}
 (\t{\X}_1^2 + \t{\X}_2^2 + \t{\X}_3^2) \t{d}^2_{\pi(0)}(\pi((0,\tau,0,0))) = (\X_1^2 + \X_2^2 + \X_3^2) d_0^2((0,\tau,0,0)).
\end{align}
As a consequence, the same argument as in the proof of Theorem \ref{tloc} proves that $\t{M}$ does not satisfy the $\MCP(K,N)$ for $K \in \R$ and $N\in (1,\infty)$.
\end{proof}

%\nocite{*}
\bibliographystyle{abbrv}
\bibliography{CSCSCbib}

\end{document}